\documentclass[11pt]{amsart}

\usepackage{amscd}
\usepackage{amsmath, amssymb,color}
\usepackage{amsfonts}

\newcommand{\de}{\partial}

\newcommand{\vp}{\varphi}
\newcommand{\ve}{\varepsilon}

\def\diam{\mathrm{diam}}
\newcommand\R{\mathbb R^n}
\def\la{\lambda}
\def\Om{\Omega}
\def\Sym{\mathrm{Sym}}

\begin{document}
\newcounter{remark}
\newcounter{theor}
\setcounter{remark}{0}
\setcounter{theor}{1}
\newtheorem{claim}{Claim}
\newtheorem{theorem}{Theorem}[section]
\newtheorem{lemma}[theorem]{Lemma}
\newtheorem{corollary}[theorem]{Corollary}
\newtheorem{proposition}[theorem]{Proposition}
\newtheorem{question}{Question}[section]
\newtheorem{defn}[theorem]{Definition}
\newtheorem{remark}[theorem]{Remark}

\numberwithin{equation}{section}

\title{Liouville theorem for a class of Hessian equations}

\author{Jianchun Chu}
\address{School of Mathematical Sciences, Peking University, Yiheyuan Road 5, Beijing, P.R.China, 100871}
\email{jianchunchu@math.pku.edu.cn}

\author{S\l awomir Dinew}
\address{Faculty of Mathematics and Computer Science, Jagiellonian University 30-348 Krakow, {\L}ojasiewicza 6, Poland}
\email{slawomir.dinew@im.uj.edu.pl}

\begin{abstract}
In this paper, we study a general class of Hessian elliptic equations, including the Monge-Amp\`ere equation, the $k$-Hessian equation and $p$-Monge-Amp\`ere equations. We propose new additional condition on the solution and prove Liouville theorem under this assumption. We show that our general condition covers as special cases numerous sets of assumptions known in the literature which were tailored for specific equations. Thus we obtain a significant generalization of multiple isolated Liouville theorems and conditional interior estimates.
\end{abstract}

\maketitle

\section{Introduction}\label{introduction}
A classical result due to J\"orgens, Calabi and Pogorelov (see \cite{J,Ca,P} and \cite{CYa} for an alternative proof) deals with entire convex solutions to the equation
\begin{equation}\label{MA}
	\det\left(\frac{\de^2 u}{\de x_i\de x_j}\right)=1\ {\rm in}\ \mathbb R^n.
\end{equation}
The theorem states that any smooth convex solution of (\ref{MA}) has to be quadratic polynomial. This result was later extended by Caffarelli and Li to {\it viscosity} solutions - see \cite{CL}.

It is natural to try to generalize J\"orgens-Calabi-Pogorelov's theorem in a different direction and study the corresponding problem for other types of Hessian operators . It turns out that the problem is highly non-trivial and involves both the algebraic structure of the operator and the convexity properties of the {\it admissible} solutions. With the notable exception of the so called special Lagrangian equations with large phase (see \cite{Y2}), we are not aware of Hessian type operators for which such a theorem holds {\it unconditionally}.

Hence a question arises what are the {\it optimal} assumptions leading to such a Liouville type result for the problem
\begin{equation}\label{Hess}
F(D^2u)=1 \ {\rm in}\ \mathbb R^n,
\end{equation}
where $F$ is a nonlinear Hessian operator. There is a vast literature on the subject - see \cite{Y1,BCGJ,WY,CYu,LRW,CX,D,SY,WB,CDH,LR} and the references therein for a sample of results  under varying assumptions.

When it comes to the assumptions on $F$ {\it ellipticity} on a suitable cone of functions is indispensable. As virtually all known arguments rely on the Evans-Krylov theory the {\it concavity} of $F$ (or a suitable modification thereof) is also incorporated into the standard set of assumptions. As most Hessian type elliptic operators are {\it homogeneous}, this is also often assumed.

Moving to the conditions imposed on the solution itself, the most frequent ones are {\it quadratic growth}, i.e. one assumes a priori that the solution $u$ to the problem
satisfies $u(x)\geq C^{-1}|x|^2-C$ for some fixed constant $C>0$. This is a natural assumption that is related to blow-up profiles when one rescales such a Hessian equation around a fixed point. A much stronger condition is the {\it double quadratic growth}, i.e.
$$ C^{-1}|x|^2-C\leq u(x)\leq C|x|^2+C.$$
The latter one, coupled with suitable {\it interior a priori estimates} is often sufficient to obtain the corresponding Liouville theorem - see \cite{D,WB,CDH}. Needless to say, it is much harder to check in practice.

Another type of condition is the a priori assumption that $u$ is convex or at least has some stronger convexity properties than the generic admissible function associated to the operator $F$. The geometric nature of such an assumption simplifies a lot the reasoning - see \cite{BCGJ}. Again, it is a very restrictive condition.

The main goal of this paper is to introduce new additional assumption - condition D (see Definition \ref{conD}) on the solution. Our assumption is sufficiently general to cover virtually all such previous conditions in the literature and is still sufficient to establish crucial interior gradient and Hessian estimates. This estimate, coupled with higher order bounds, provides Liouville type theorem for solutions with quadratic growth.

Below we state our main technical device. We refer the reader to Section \ref{preliminaries} for all the relevant definitions.
\begin{theorem}\label{intgrad}
Assume that $\Gamma$ is a convex invariant cone and $F$ is a concave elliptic Hessian operator satisfying condition N on $\Gamma$. Fix a bounded domain $\Omega$ in $\R$ and let the smooth admissible function $u$ solve the problem
	\begin{equation}\label{key1}
		\begin{cases}
			\ F(D^{2}u) = 1 & \mbox {in $\Omega$}, \\
			\ u = 0 & \mbox{on $\de\Omega$}.
		\end{cases}
\end{equation}
Assume furthermore that the function $u$ satisfies condition D. Then there exist constants $\alpha$, $\beta$ and $C$ depending only on $F$, $N_{1}$, $N_{2}$, $D_{1}$, $D_{2}$ and $\diam(\Omega)$ such that
\[
\sup_{\Omega}\left((-u)+(-u)^{\alpha}|Du|+(-u)^{\beta}|D^{2}u|\right) \leq C.
\]
\end{theorem}

Some comments are in order here. A gradient estimate by Chou and Wang - (\cite{CW}, see also \cite{Tr} for an alternative proof) shows that for the $k$- Hessian operator one has an a priori $C^1$ interior bound
$$|Du(x)|\leq \frac{C}{{\mathrm{dist}}(x,\de\Omega)},\ \ x\in\Omega$$
for a constant $C$ depending only on the oscillation $\mathrm{osc}(u)$, $k$ and $n$. The major difference is that the distance to the boundary is measured by the Euclidean distance, whereas in our setting it is measured by the decay of the solution $u$. Gradient estimates for other Hessian equations mimicking Chou-Wang's result can be found in \cite{D,Ch}. To the contrary, almost nothing is known about interior gradient estimates based on the decay of the solution (without imposing additional assumptions). In fact such bounds are known only under very restrictive additional conditions such as convexity - see \cite{BCGJ}.

The difference may seem insignificant at first sight but the gradient estimate is then coupled with higher order bounds to obtain the Liouville theorem. A second order estimate of Chou and Wang - \cite{CW} (still for the $k$-Hessian equation) states that $C^4$-admissible solution $u$ to (\ref{key1}), vanishing on $\de\Omega$ obeys the bound
$$(-u(x))^\beta|D^2 u(x)|\leq C,\ \ x\in\Omega$$
with $\beta>1$ and $C$, depending only on $k$, $n$ and the gradient bound of $u$. This estimate has also been shown to work for many Hessian equations (see \cite{LRW,CX,D,CDH,LR}). In turn, an unconditional interior second order estimate depending on the Euclidean distance to the boundary seems out of reach in general. In this direction, one of the most interesting works is the recent paper \cite{GQ} where such a bound was proven for $2$-Hessian equation under the additional assumption that $\sigma_3(D^2u)\geq-A$ for a positive constant $A$.

Coupling the gradient and Hessian estimates with a suitable Evans-Krylov $C^{2,\gamma}$ estimates is then sufficient to prove the Liouville property if double quadratic growth is assumed - see \cite{WB}. The argument breaks down if mere quadratic growth is assumed - the problem is exactly that the different types of estimates of first and second order do not pair well and the oscillation of a (suitably rescaled) $u$ in a unit ball becomes uncontrollable.

By Theorem \ref{intgrad}, standard argument leads to the following theorem:
\begin{theorem}\label{lio}
Assume that $\Gamma$ is a convex invariant cone and $F$ is a concave elliptic Hessian operator satisfying condition N on $\Gamma$. If an entire smooth admissible function $u$ has quadratic growth, solves
$$F(D^2u)=1\ {\rm in}\ \R.$$
and furthermore satisfies condition D, then it is a quadratic polynomial.
\end{theorem}

As an application, we now list the following corollaries:

\begin{corollary}\label{Liouville k Hessian}
Let $u$ be a smooth entire $k$-convex function. If $u$ has quadratic growth, solves
$$\sigma_k(D^2u)=1\ {\rm in}\ \mathbb R^n.$$
and is furthermore $(k+1)$-convex, then it is a quadratic polynomial. The statement still holds if there is a uniform constant $A$ such that for any $x\in \R$, $\sigma_{k+1}(D^2u(x))\geq-A$.
\end{corollary}

\begin{corollary}\label{Liouville p MA}
Let $u$ be a smooth entire $p$-plurisubharmonic function. If $u$ has quadratic growth, solves
$$\mathcal M_p(D^2u)=1\ {\rm in}\ \mathbb R^n.$$
and is furthermore $(p-1)$-plurisubharmonic, then it is a quadratic polynomial. The statement still holds if there is a uniform constant $A$ such that for any $x\in\R$,
$$\la_{i_1}(D^2u(x))+\cdots+\la_{i_{p-1}}(D^2u(x))\geq -A$$
for any $(p-1)$-tuple of indices $1\leq i_1<i_2<\cdots<i_{p-1}\leq n$.
\end{corollary}

When $u$ is assumed to be $(k+1)$-convex, Corollary \ref{Liouville k Hessian} was proved by Li, Ren and Wang - see \cite{LRW}. When $k=2$, Corollary \ref{Liouville k Hessian} was proved by Chen and Xi - see \cite{CX}. The key ingredient in both papers are the a priori estimates (i.e. Theorem \ref{intgrad}). Note that our method is totally different from that of \cite{LRW,CX}, where they established $C^{2}$ estimate directly and the structure of $\sigma_{k}^{1/k}$ plays a crucial role. The assumption of Theorem \ref{intgrad} is quite general. It can be shown that $\sigma_{k+1}(D^2u(x))\geq-A$ implies that $u$ satisfies condition D - see Proposition \ref{khess}. In the proof of Theorem \ref{intgrad}, combining condition D and a series of delicate calculations, we first establish the gradient estimate, then use it as part of auxiliary function and finally prove $C^{2}$ estimate. Such method not only works for a general class of Hessian equations, but also simplifies the argument in some extent and allows us to generalize the results of \cite{LRW,CX}.

\bigskip

The note is organized as follows: in Section \ref{preliminaries}, we gather the useful notions and classical results. The Condition D and its relation to previous assumptions is discussed in Section \ref{condition D}. The proof of Theorem \ref{intgrad} is given in \ref{main estimate}. In Section \ref{Liouville theorem}, we prove Theorem \ref{lio} and its corollaries.

\bigskip

{\bf Acknowledgements.}
The first-named author was partially supported by the Fundamental Research Funds for the Central Universities, Peking University. The second-named author was partially supported by grant no. 2021/41/B/ST1/01632 from the National Science Center, Poland.

\section{Preliminaries}\label{preliminaries}
Here we gather the necessary definitions and results needed later on.

\subsection{Convex invariant cones and admissible functions}
By $\Sym^2(\R)$ we denote the space of symmetric $n\times n$ matrices - the space where the pointwise Hessians of $C^2$ functions live. For any $A\in \Sym^2(\R)$, by $\la(A):=(\la_1(A),\cdots,\la_n(A))$ we denote the vector of eigenvalues (the ordering of $\la_i(A)$'s, unless otherwise stated, is arbitrary). If the choice of the matrix $A$ in question is clear from the context we shall also use the notation $\la=(\la_1,\cdots,\la_n)$.

Fix $\Gamma\subset\R$ - an open cone with vertex at the coordinate origin. The basic example of these are the cones
\begin{equation}\label{gammak}
\Gamma_k:=\lbrace \la=(\la_1,\cdots,\la_n) \ | \ \sigma_j(\la)>0,\ j=1,\cdots,k\rbrace,
\end{equation}
where $\sigma_j$ is the $j$-th elementary symmetric polynomial on the coordinates:
\[
\sigma_j(\la):=\sum_{1\leq i_1<i_2<\cdots<i_j\leq n}\la_{i_1}\cdots\la_{i_j}.
\]
Note that $\Gamma_n$ is the positive orthant, while $\Gamma_1$ is a half-space containing $\Gamma_n$. Below we recall the notion of an invariant cone:
\begin{defn}\label{ST}
An open cone $\Gamma\subset\R$ is said to be invariant if the following conditions hold:
\begin{itemize}\setlength{\itemsep}{1mm}
\item (Positivity) $\Gamma+\Gamma_n\subset\Gamma$;
\item (Invariance) $\Gamma$ is invariant under permutation of coordinates.
\end{itemize}
If furthermore $\Gamma$ is convex, we call it a convex invariant cone.
\end{defn}
\begin{remark} Harvey and Lawson in \cite{HL} considered the so-called ST-subequations. These are the subcones $\mathcal C$ in $\Sym^2(\R)$ that are invariant under a transitive subgroup of $O(n)$ and satisfy $\mathcal C+\mathcal P\subset\mathcal P$, where $\mathcal P$ is the cone of positive definite symmetric matrices. It is straightforward to see that for any such ST-subequation $\mathcal C$ the set
\[
\la(\mathcal C):=\lbrace\la(A)\in\mathbb{R}^{n} \ | \ A\in\mathcal C\rbrace
\]
is an invariant cone in $\R$. Thus the Harvey-Lawson theory of ST-subequations  provides a rich source of examples of invariant cones.
\end{remark}
All the cones we shall consider will be convex invariant cones. Note that this forces the inclusion
$\Gamma\subset\Gamma_1$. It is in particular well-known that all cones $\Gamma_k$, $k=1,\cdots,n$ are of this type. For more examples we refere to \cite{HL}, where the basic theory of such cones is discussed.

Coupled with such cones is the notion of an {\it admissible function}:
\begin{defn}\label{admissible}
Let $\Gamma$ be a convex invariant cone. A $C^2$-smooth function $u$ is said to be admissible (with respect to $\Gamma$) if for any $x$ in its domain, the Hessian $D^2u(x)$ has a vector of eigenvalues belonging to $\Gamma$.
\end{defn}

Note that by the invariance it doesn't matter what the order of the eigenvalues is. A simple observation is that admissible functions (for a fixed $\Gamma$) form themselves a convex cone. The inclusion $\Gamma\subset\Gamma_1$ implies that admissible functions are subharmonic, i.e. $\Delta u>0$.

\subsection{Elliptic Hessian operator}
Next we define a class of Hessian operators. These are the operators that depend only on the eigenvalues of the Hessian of the function and are {\it elliptic} whenever restricted to the class of admissible functions.
\begin{defn}\label{Hesoper}
Let $\Gamma$ be a convex invariant cone. The map
\[
F: \lbrace A\in \Sym^2(\R)\ |\ \la(A)\in \Gamma\rbrace \longmapsto \mathbb R
\]
is said to be an elliptic Hessian operator if the following holds:
\begin{enumerate}\setlength{\itemsep}{1mm}
\item (Dependence on eigenvalues only): there is a $C^2$ smooth function $f:\Gamma\longmapsto \mathbb R$ such that
$$F(A)=f(\la_1(A),\cdots,\la_n(A)),$$
where $\la_i(A)$ denote the eigenvalues of $A$;
\item (Ellipticity) $\frac{\de f}{\de\lambda_{i}}(\lambda) > 0, \ \lambda\in\Gamma, \ i=1,\cdots,n.$

\item (Definiteness) $F(A)>0$ for symmetric matrices $A$ with vector of eigenvalues in ${\Gamma}$.

\item  (Homogeneity) $F(tA)=tF(A)$ whenever $t>0$ and $A$ has an eigenvalue vector in $\Gamma$.
\end{enumerate}	
\end{defn}

\begin{remark}\label{remark}
As we already mentioned above, the homogeneity is a typical feature of most Hessian operators of interest but is sometimes excluded from the definition. We have included it into the definition as all the operators we deal with share this property. Note also that the ellipticity condition is imposed on $f$, but it implies that $F(A+P)>F(A)$, whenever $A$ and $P$ are symmetric matrices with vector of eigenvalues in $\Gamma$ and $\Gamma_n$, respectively.
\end{remark}
Examples of pairs $(\Gamma,F)$ include:
\begin{itemize}\setlength{\itemsep}{2mm}
	\item $(\Gamma_n,\det(A)^{1/n})$ - the Monge-Amp\`ere operator;
	\item $(\Gamma_k, \sigma_k(\la(A))^{1/k})$ - the $k$-Hessian operator;
	\item $(\Gamma_k, \left(\frac{\sigma_k(\la(A))}{\sigma_l(\la(A))}\right)^{1/(k-l)})$ for $1\leq l<k\leq n$ - the Hessian quotient operator;
	\item $(\hat{\Gamma}_p,\mathcal M_p(A)^{1/\binom np})$ - the $p$-Monge-Amp\`ere operator,
	where
	$$\hat{\Gamma}_p:=\lbrace \la\in\R\ |\  \la_{i_1}+\cdots+\la_{i_p}>0 \ \text{for any $1\leq i_1<\cdots<i_p\leq n$} \rbrace$$
	 and
$$\mathcal M_p(A)=\Pi_{1\leq i_1<\cdots<i_p\leq n}\left(\la_{i_1}(A)+\cdots+\la_{i_p}(A)\right).$$
\end{itemize}
We refer to \cite{CNS} and especially to \cite{HL} for many more such examples.

We also recall the following nice property of homogeneous operators:
\begin{lemma}\label{homog}
Assume that $\Gamma$ is a convex invariant cone and $F$ is an elliptic Hessian operator on $\Gamma$. If $\lambda\in\Gamma$, then
\begin{equation}\label{homog eqn}
\sum_{i=1}^n \frac{\de f}{\de\lambda_{i}}(\lambda)\la_i = f(\la).
\end{equation}
\end{lemma}
\begin{proof}
For $t>0$, one has $f(t\la)=tf(\la)$. Differentiating this equation with respect to $t$ and evaluating at $t=1$ yields the claim.
\end{proof}
For Hessian equations, uniform (i.e. $C^0$) estimates follow from the maximum principle. In order to obtain global higher order estimates for the problem (\ref{key1}), a maximum principle type argument is exploited coupled with suitable boundary estimates. When it comes to {\it interior} estimates, especially if no boundary regularity is assumed, the problem becomes very subtle and crucially depends on the algebraic properties of the operator $F$ - see \cite{CW,D,Ch}. For Hessian and higher order bounds, the concavity of the operator becomes very helpful and is quite often incorporated into the standard set of assumptions made on $F$.
\begin{defn}\label{concave}
The elliptic Hessian operator $F: \Sym^2(\R)\longmapsto \mathbb R$ is said to be concave if
$F$ is a concave function on the set of matrices
$$\lbrace A\in \Sym^2(\R)\ |\ \la(A)\in\Gamma\rbrace.$$
\end{defn}
It is well-known that all the examples mentioned after Remark \ref{remark} are concave Hessian operators.

\begin{remark}
Write $A=(a_{ij})$, $f_{i}=\frac{\de f}{\de\lambda_{i}}$, $f_{ij}=\frac{\de^{2}f}{\de\lambda_{i}\de\lambda_{j}}$ and
\[
F^{ij} = \frac{\de F}{\de a_{ij}}, \ F^{ij,kl} = \frac{\de^{2}F}{\de a_{ij}\de a_{kl}}.
\]
If $A$ is diagonal, i.e. $a_{ij}=\delta_{ij}\lambda_{i}$, then (see e.g. \cite{A,G,S})
\begin{equation}\label{derivative of F}
F^{ij} = \delta_{ij}f_{i}, \ \ F^{ij,kl} = f_{ik}\delta_{ij}\delta_{kl}+\frac{f_{i}-f_{j}}{\lambda_{i}-\lambda_{j}}(1-\delta_{ij})\delta_{il}\delta_{jk},
\end{equation}
where the quotient is interpreted as a limit if $\lambda_{i}=\lambda_{j}$. Hence, the concavity of $F$ is equivalent to the concavity of $f$.
\end{remark}

Finally we introduce a technical condition on the operator $F$ which we dub as {\it Condition N}.

\begin{defn}[Condition N]
The elliptic Hessian operator $F$ is said to satisfy condition N if the following holds: whenever $\la\in\Gamma$ is a vector such that
$f(\la)=1$ and $\sum_{i=1}^n\frac{\de f}{\de \la_i}(\la)\leq C$ for some constant $C>0$, then
\[
\frac{\de f}{\de \la_i}(\la)\geq \frac{1}{N_{1}C^{N_{2}}}\ \text{for any $i=1,\cdots,n$}.
\]
for some constants $N_{1}$ and $N_{2}$ depending only on $(\Gamma,F)$.
\end{defn}

Heuristically Condition N allows one to bound the coefficients of the linearized operator associated to $F$ from below once an estimate from above is available. This is a typical scenario for all a priori estimates whenever a maximum principle argument is applied to the linearized operator.

The following proposition shows that a mild additional assumption on the ellipticity of $F$ leads to the condition N.

\begin{proposition}\label{k Hessian p MA condition N}
Assume that $\Gamma$ is a convex invariant cone and $F$ is a concave elliptic Hessian operator on $\Gamma$. If additionally $f$ satisfies
$$ f(\la+\tau)\geq f(\la)+d(\tau_1\cdots\tau_n)^{1/n}, \ \ \la\in\Gamma,\ \tau\in\Gamma_n\ $$
for some constant $d>0$, then the condition N holds for $F$.
\end{proposition}
\begin{proof}
Fix $\la\in\Gamma$ and let $\tau\in\Gamma_n$ to be chosen later on. Assume that $f(\la)=1$ and $\sum_{i=1}^n\frac{\de f}{\de \la_i}(\la)\leq C$ for some fixed constant $C>0$.
Then
\begin{equation}\label{k Hessian p MA condition N eqn}
\sum_{k=1}^{n}\frac{\de f}{\de \la_k}(\la)\tau_k\geq \int_{0}^{1}\sum_{k=1}^{n}\frac{\de f}{\de \la_k}(\la+s\tau)\,\tau_k \,ds=f(\la+\tau)-f(\la),
\end{equation}
where we used the concavity to justify the inequality. The latter quantity is bounded from below by $d(\tau_1\cdots\tau_n)^{1/n}$. Now fix $i\in\lbrace1,\cdots,n\rbrace$ and define
$\tau_i:=C$, $\tau_j:=\frac{\de f}{\de \la_i}(\la)$ for $j\neq i$. Then \eqref{k Hessian p MA condition N eqn} becomes
\[
2C\frac{\de f}{\de \la_i}(\la)
\geq \frac{\de f}{\de \la_i}(\la)\la_i+\sum_{k\neq i}\frac{\de f}{\de \la_k}(\la)\la_k
= \sum_{k=1}^n\frac{\de f}{\de \la_k}(\la)\la_k\geq dC^{\frac{1}{n}}\left(\frac{\de f}{\de \la_i}(\la)\right)^{\frac{n-1}{n}}.
\]
Reordering the terms, we obtain
$$\frac{\de f}{\de \la_i}(\la)\geq\frac1{(2/d)^nC^{n-1}},$$
which finishes the proof.
\end{proof}

\begin{remark}\label{condition N remark}
The domination of the product of $\la_i$ is a much weaker condition than strict ellipticity. In a very recent paper \cite{HL2}, Harvey and Lawson have shown that for any pair $(\la(G), F)$ with $F$ being an invariant G\aa rding-Dirichlet polynomial and $G$ being the component of $\Sym^2(\R)\setminus\lbrace F=0\rbrace$ containing the identity matrix, we have $f(\la+\tau)\geq f(\la)+d(\tau_1\cdots\tau_n)^{1/n}$, $\la\in\Gamma$, $\tau\in\Gamma_n$  and hence $F$ satisfies condition N. In particular, all the explicit examples above, with the notable exception of the Hessian quotient operator, satisfy condition N - see \cite[Section 3]{HL2}.
\end{remark}

\section{Condition D}\label{condition D}

Below we define the Condition D that is advertised in Section \ref{introduction}.
\begin{defn}[Condition D]\label{conD}
Assume that $\Gamma$ is a convex invariant cone and $F$ is a concave elliptic Hessian operator on $\Gamma$. An admissible function $u$ is said to satisfy condition D if the following holds: $F(D^2u)=1$ and for any point $x$ and index $i\in\lbrace1,\cdots,n\rbrace$, if the eigenvalue $\la_i=\la_i(D^2u(x))$ is larger than $D_{1}$, then
$$\frac{\de f}{\de \la_i}(\la)\la_i\leq D_{2}$$
for some constants $D_{1}$ and $D_{2}$ depending only on $(\Gamma,F)$.
\end{defn}

Heuristically condition D says that whenever any eigenvalue is large, the corresponding coefficient of the linearized operator is small enough to tame the size of the product. Keeping Lemma \ref{homog} in mind, condition D states that the summands in \eqref{homog eqn} corresponding to large eigenvalues do not deviate too much from the average.

The condition D is of analytic nature and clearly depends on the operator $F$ through the function $f$. Below we state a purely geometric condition, inspired by \cite{CNS}.
	
\begin{defn}[Condition CNS]\label{conCNS}
Let $\Gamma$ be a convex invariant cone. An admissible function $u$ is said to satisfy condition CNS if the following holds: there is a uniform constant $R>0$ such that for any $x\in\Om$ and $i\in\lbrace1,\cdots,n\rbrace$ the vector
$$\left(\la_1(D^2u(x)),\cdots,\la_{i-1}(D^2u(x)),R,\la_{i+1}(D^2u(x)),\cdots,\la_n(D^2u(x))\right)$$
belongs to $\Gamma$.	
\end{defn}

\begin{remark}
It is clear that all convex $C^2$ smooth functions satisfy condition CNS.
\end{remark}

The following proposition links the above two notions for concave Hessian operators:
\begin{proposition}\label{CNSimplyD}
Assume that $\Gamma$ is a convex invariant cone and $F$ is a concave elliptic Hessian operator on $\Gamma$. If an admissible function $u$ satisfies $F(D^{2}u)=1$ and condition CNS, then $u$ satisfies condition $D$ with $D_{1}=2R$ and $D_{2}=2$.	
\end{proposition}

\begin{proof}
Fix an admissible function $u$ and a point $x$ in its domain of definition. Suppose that $u$ satisfies condition CNS with parameter $R$. Let the eigenvalues of $D^2u(x)$ be such that $f(\la)=1$, while $\la_n>2R$. Then $(\lambda',R)\in\Gamma$ for $\lambda'=(\lambda_{1},\cdots,\lambda_{n-1})$ and we compute
\[
f(\la) = f(\la',R)+\int_{R}^{\la_n}\frac{\de f}{\de \la_n}(\lambda',s)ds\geq \frac{\de f}{\de \la_n}(\la)(\la_n-R),
\]
by concavity. Hence
$$1\geq \frac12\frac{\de f}{\de \la_n}(\la)\la_n,$$
which yields the claim.
\end{proof}

Let us now compare condition D with other convexity assumptions made in the literature. As we have already mentioned, if an admissible function is {\it convex} then is satisfies condition CNS and hence condition D (this is also obvious from Lemma \ref{homog}). The main point however is that many weaker convexity conditions imply condition D.

For the $p$-Monge-Amp\`ere equation, the following proposition is trivial.
\begin{proposition}\label{p-1}
Let $(\Gamma,F)$ be the $p$-Monge-Amp\`ere equation on $\hat{\Gamma}_p$. The admissible functions then are characterized by the property that the sum of any $p$ distinct eigenvalues of the Hessian at any point $x_0$ is positive (the $p$-plurisubharmonicity). Then any $p$-plurisubharmonic function which is $(p-1)$ plurisubharmonic satisfies the condition CNS. In fact, it suffices to assume that there is $R$ such that $u(x)+\frac{R}{2}|x|^2$ is $(p-1)$-plurisubharmonic.
\end{proposition}

For the $k$-Hessian equation, we have the following proposition.

\begin{proposition}\label{khess}
Let $(\Gamma,F)$ be the $k$-Hessian equation on ${\Gamma}_k$. The admissible functions then are characterized by the property that $\sigma_j(D^2u(x))\geq 0$ for $j=1,\cdots,k$ (the $k$-convexity). Then any $k$-convex function which is $(k+1)$ convex satisfies the condition CNS. In fact, it suffices to assume that there is $A$ such that $u$ is $k$-convex and $\sigma_{k+1}(D^2u(x))\geq -A\sigma_k(D^2u(x))$.
\end{proposition}

\begin{proof}
Fix $\la=(\lambda',\lambda_{n})\in\Gamma_k$. It suffices to show that $\sigma_{k+1}(\la)>-A$ implies that $(\la',R)\in\Gamma_k$ for $R$ depending only on $A$.

Denote by $\sigma_j(\la|i)=\frac{\de \sigma_j}{\de \la_i}(\la)$. Note that this quantity does not depend on $\la_i$. It is well-known that
\begin{equation}\label{sik}
\sigma_j(\la)=\sigma_{j-1}(\la|i)\la_i+\sigma_j(\la|i), \ \ j=1,\cdots,n.
\end{equation}
and
\begin{equation}\label{elliptic sigma k}
 \sigma_j(\la|i) > 0, \ \ j=1,\cdots,k-1.
\end{equation}
Using \eqref{sik}, \eqref{elliptic sigma k} and an induction argument, we see that $(\lambda',R)\in\Gamma_{k-1}$. To prove $(\la',R)\in\Gamma_k$, it then suffices to show $\sigma_{k}(\la',R)>0$.

Recall the classical Newton inequality:
\begin{equation}\label{Newton}
\left[\frac{\sigma_j(\mu)}{\binom n j}\right]^2\geq 	\frac{\sigma_{j-1}(\mu)}{\binom n {j-1}}\cdot\frac{\sigma_{j+1}(\mu)}{\binom n {j+1}},\ \ \mu\in\R, \ j=1,\cdots,n-1.
\end{equation}
It is worth emphasizing that \eqref{Newton} holds for {\it any} vector $\mu\in\R$ and not only for those in $\Gamma_{j+1}$. It is clear that
\begin{equation}\label{weakNewton}
[{\sigma_j(\mu)}]^2
\geq  \frac{{\binom n {j-1}}{\binom n {j+1}}}{{\binom n j}^{2}}[{\sigma_j(\mu)}]^2
\geq {\sigma_{j-1}(\mu)}{\sigma_{j+1}(\mu)}.
\end{equation}
Applying \eqref{weakNewton} to $\lambda'=(\lambda_{1},\cdots,\lambda_{n-1})$, we obtain
$$[{\sigma_{k}(\la')}]^2\geq {\sigma_{k-1}(\la')}{\sigma_{k+1}(\la')},$$
which can be rewritten as
$$[{\sigma_{k}((\la',R)|n)}]^2\geq {\sigma_{k-1}((\la',R)|n)}\sigma_{k+1}((\la',R)|n).$$
Combining this with (\ref{sik}) and $\sigma_{j}(\lambda|n)=\sigma_{j}((\la',R)|n)$,
\begin{equation}\label{kkk}
\begin{split}
& \sigma_{k}((\la',R)|n)\sigma_k(\la) \\
= {} & \sigma_{k}((\la',R)|n)\left(\sigma_k(\la|n)+\lambda_{n}\sigma_{k-1}(\la|n)\right) \\
= {} & \sigma_{k}((\la',R)|n)\left(\sigma_k((\la',R)|n)+\lambda_{n}\sigma_{k-1}((\la',R)|n)\right) \\
\geq {} & {\sigma_{k-1}((\la',R)|n)}\sigma_{k+1}((\la',R)|n)+\lambda_{n}\sigma_{k-1}((\la',R)|n)\sigma_{k}((\la',R)|n) \\
= {} & {\sigma_{k-1}((\la',R)|n)}\left(\sigma_{k+1}(\la|n)+\lambda_{n}\sigma_{k}(\la|n)\right) \\
= {} & {\sigma_{k-1}((\la',R)|n)}\sigma_{k+1}(\lambda).
\end{split}
\end{equation}
From $\la\in\Gamma_k$, we obtain $\sigma_k(\la)>0$ and
$$\sigma_{k-1}((\la',R)|n) = \sigma_{k-1}(\la|n) = \frac{\de \sigma_k}{\de \la_n}(\la)>0.$$
Then (\ref{kkk}) becomes
\begin{equation}\label{kk}
\frac{\sigma_{k}((\la',R)|n)}{\sigma_{k-1}((\la',R)|n)}\geq\frac{\sigma_{k+1}(\la)}{\sigma_k(\la)}>-A.
\end{equation}
Finally picking any $R>A$ and once again exploiting (\ref{sik}) one obtains
$$\sigma_k(\la',R)=\sigma_{k}((\la',R)|n)+R\sigma_{k-1}((\la',R)|n)>0,$$
which finishes the proof.
\end{proof}

\section{Main estimate}\label{main estimate}
In this section, we give the proof of Theorem \ref{intgrad}. Without loss of generality, we assume that $0\in\Omega$ for convenience.

\subsection{$C^{0}$ estimate}

\begin{proposition}\label{C0 estimate}
Under the same assumptions of Theorem \ref{intgrad}, there exists a constant $C$ depending only on $F$ and $\diam(\Omega)$ such that
\[
-C \leq u < 0 \ \text{in $\Omega$}.
\]
\end{proposition}

\begin{proof}
From $\Gamma\subset\Gamma_{1}$, we see that $\Delta u>0$. Then the strong maximum principle and $u=0$ on $\de\Omega$ show
\[
u < 0 \ \text{in $\Omega$}.
\]
For the lower bound of $u$, we define
\[
\underline{u} = \frac{A}{2}\left(|x|^{2}-\diam^{2}(\Omega)\right).
\]
Writing $\mathbf{1}=(1,\cdots,1)$ and choosing $A=1/f(\mathbf{1})$, we have
\[
\underline{u} \leq 0 \ \text{and} \ F(D^{2}\underline{u}) = f(A\mathbf{1}) = Af(\mathbf{1}) = 1
\]
and so
\[
\begin{cases}
\ F(D^{2}\underline{u}) = F(D^{2}u) & \mbox {in $\Omega$}, \\
\ \underline{u} \leq u & \mbox{on $\de\Omega$}.
\end{cases}
\]
By the comparison principle, we obtain $\underline{u}\leq u$ in $\Omega$.
\end{proof}

\subsection{$C^{1}$ estimate}

\begin{proposition}\label{C1 estimate}
Under the same assumptions of Theorem \ref{intgrad}, there exist constants $\alpha$ and $C$ depending only on $\|u\|_{C^{0}(\Omega)}$, $N_{1}$, $N_{2}$, $D_{1}$, $D_{2}$ and $\diam(\Omega)$ such that
\[
\sup_{\Omega}\Big((-u)^{\alpha}|Du|\Big) \leq C.
\]
\end{proposition}

\begin{proof}
We consider the quantity
\[
Q = \frac{1}{2}\log|Du|^{2}+\alpha\log(-u)+\frac{u^{2}}{2}+\frac{|x|^{2}}{2},
\]
where $\alpha$ is a positive constant to be determined later. Let $x_{0}$ be the maximum point of $Q$. We choose the coordinate system such that
\[
u_{ij}(x_{0}) = \delta_{ij}\lambda_{i}.
\]

\begin{lemma}
At $x_{0}$, we have
\begin{equation}\label{Q i}
0 = \frac{u_{i}\lambda_{i}}{|Du|^{2}}+\frac{\alpha u_{i}}{u}+uu_{i}+x_{i}
\end{equation}
and
\begin{equation}\label{Q ii}
0 \geq \frac{F^{ii}\lambda_{i}^{2}}{|Du|^{2}}-\frac{3F^{ii}u_{i}^{2}\lambda_{i}^{2}}{|Du|^{4}}
+\left(\frac{\alpha}{u}+u\right)F^{ii}\lambda_{i}+\frac{F^{ii}u_{i}^{2}}{2}+\frac{1}{2}\sum_{i}F^{ii}.
\end{equation}
\end{lemma}

\begin{proof}
The maximum principle shows $Q_{i}=0$ and $F^{ii}Q_{ii}\leq0$ at $x_{0}$. Then we obtain \eqref{Q i} and
\[
\begin{split}
0 \geq {} & \frac{F^{ii}u_{iik}u_{k}}{|Du|^{2}}+\frac{F^{ii}u_{ii}^{2}}{|Du|^{2}}-\frac{2F^{ii}u_{i}^{2}u_{ii}^{2}}{|Du|^{4}} \\
& +\frac{\alpha F^{ii}u_{ii}}{u}-\frac{\alpha F^{ii}u_{i}^{2}}{u^{2}}+F^{ii}u_{i}^{2}+uF^{ii}u_{ii}+\sum_{i}F^{ii}.
\end{split}
\]
Differentiating the equation $F(D^{2}u)=1$ with respect to $x_{k}$, we see that $F^{ii}u_{iik}=0$ and so
\begin{equation}\label{LQ}
0 \geq \frac{F^{ii}\lambda_{i}^{2}}{|Du|^{2}}-\frac{2F^{ii}u_{i}^{2}\lambda_{i}^{2}}{|Du|^{4}}
-\frac{\alpha F^{ii}u_{i}^{2}}{u^{2}}+\left(\frac{\alpha}{u}+u\right)F^{ii}\lambda_{i}+F^{ii}u_{i}^{2}+\sum_{i}F^{ii}.
\end{equation}
It follows from \eqref{Q i} that
\[
\begin{split}
\frac{u_{i}^{2}}{u^{2}} = {} & \frac{1}{\alpha^{2}}\left(\frac{u_{i}\lambda_{i}}{|Du|^{2}}+uu_{i}+x_{i}\right)^{2} \\
\leq {} & \frac{3}{\alpha^{2}}\left(\frac{u_{i}^{2}\lambda_{i}^{2}}{|Du|^{4}}+u^{2}u_{i}^{2}+x_{i}^{2}\right) \\
\leq {} & \frac{3}{\alpha^{2}}\left(\frac{u_{i}^{2}\lambda_{i}^{2}}{|Du|^{4}}+Cu_{i}^{2}+C\right),
\end{split}
\]
which implies
\[
\frac{\alpha F^{ii}u_{i}^{2}}{u^{2}} \leq \frac{3}{\alpha}
\left(\frac{F^{ii}u_{i}^{2}\lambda_{i}^{2}}{|Du|^{4}}+CF^{ii}u_{i}^{2}+C\sum_{i}F^{ii}\right).
\]
Substituting this into \eqref{LQ},
\[
\begin{split}
0 \geq {} & \frac{F^{ii}\lambda_{i}^{2}}{|Du|^{2}}-\left(2+\frac{3}{\alpha}\right)\frac{F^{ii}u_{i}^{2}\lambda_{i}^{2}}{|Du|^{4}}
+\left(\frac{\alpha}{u}+u\right)F^{ii}\lambda_{i} \\
& +\left(1-\frac{C}{\alpha}\right)F^{ii}u_{i}^{2}+\left(1-\frac{C}{\alpha}\right)\sum_{i}F^{ii}.
\end{split}
\]
After increasing $\alpha$ if necessary, we obtain \eqref{Q ii}.
\end{proof}

\begin{lemma}\label{C1 inequality}
At $x_{0}$, we have
\[
\frac{3F^{ii}u_{i}^{2}\lambda_{i}^{2}}{|Du|^{4}} \leq \frac{F^{ii}\lambda_{i}^{2}}{|Du|^{2}}+\frac{C\alpha}{(-u)}.
\]
\end{lemma}

\begin{proof}
Define the index set by
\[
I = \big\{i\in\{1,\cdots,n\} \ | \ 3u_{i}^{2}\leq |Du|^{2}\big\}.
\]
It follows that
\begin{equation}\label{C1 inequality eqn 1}
\sum_{i\in I}\frac{3F^{ii}u_{i}^{2}\lambda_{i}^{2}}{|Du|^{4}} \leq \sum_{i\in I}\frac{F^{ii}\lambda_{i}^{2}}{|Du|^{2}}.
\end{equation}
For $i\notin I$, we have $3u_{i}^{2}>|Du|^{2}$.  Then \eqref{Q i} shows
\[
\lambda_{i} = -\left(\frac{\alpha}{u}+u+\frac{x_{i}}{u_{i}}\right)|Du|^{2}.
\]
Increasing $\alpha$ if necessary, we obtain
\[
\frac{\alpha}{C(-u)}\cdot|Du|^{2} \leq \lambda_{i} \leq \frac{C\alpha}{(-u)}\cdot|Du|^{2}, \ i\notin I.
\]
Without loss of generality, we assume that $|Du|^{2}$ is large enough such that $\lambda_{i}\geq D_{1}$. Then condition D shows $F^{ii}\lambda_{i}\leq D_{2}$ and so
\[
\sum_{i\notin I}\frac{3F^{ii}u_{i}^{2}\lambda_{i}^{2}}{|Du|^{4}}
= \sum_{i\notin I}\left( 3F^{ii}\lambda_{i}\cdot\frac{u_{i}^{2}}{|Du|^{2}}\cdot\frac{\lambda_{i}}{|Du|^{2}} \right) \leq \frac{C\alpha}{(-u)}.
\]
\end{proof}

Now we are in a position to prove Proposition \ref{C1 estimate}. Combining \eqref{Q ii} and Lemma \ref{C1 inequality},
\[
0 \geq
\left(\frac{\alpha}{u}+u\right)F^{ii}\lambda_{i}-\frac{C}{(-u)}+\frac{F^{ii}u_{i}^{2}}{2}+\frac{1}{2}\sum_{i}F^{ii}.
\]
Using Proposition \ref{homog}, we have $\sum_{i}F^{ii}\lambda_{i}=1$ and so
\[
\sum_{i}F^{ii}u_{i}^{2}+\sum_{i}F^{ii} \leq \frac{C}{(-u)}.
\]
Using condition N, we obtain
\[
F^{ii} \geq \frac{1}{N_{1}}\left[\frac{(-u)}{C}\right]^{N_{2}}, \  i = 1,\cdots,n
\]
and so
\[
|Du|^{2} \leq N_{1}\left[\frac{C}{(-u)}\right]^{N_{2}}\sum_{i}F^{ii}u_{i}^{2} \leq \frac{C}{(-u)^{N_{2}+1}}.
\]
Increasing $\alpha$ such that $2\alpha>N_{2}+1$, we obtain Proposition \ref{C1 estimate}.
\end{proof}

\subsection{$C^{2}$ estimate}

\begin{proposition}
Under the same assumptions of Theorem \ref{intgrad}, there exist constants $\beta$ and $C$ depending only on $\|u\|_{C^{0}(\Omega)}$, $\sup_{\Omega}\left((-u)^{\alpha}|Du|\right)$, $\alpha$, $N_{1}$, $N_{2}$ and $\diam(\Omega)$ such that
\[
\sup_{\Omega}\Big((-u)^{\beta}|D^{2}u|\Big) \leq C.
\]
\end{proposition}

\begin{remark}
We emphasize that this interior Hessian bound holds without assuming condition D on $u$. What really matters though is the a priori knowledge of the interior gradient estimate $\sup_{\Omega}(-u)^{\alpha}|Du|\leq C$. On the bright side, such estimate holds for all concave Hessian operators satisfying condition N and so it provides an unified treatment. Another important advantage is that this bound is {\it purely} interior in the sense that it does not depend on the global gradient bound.
\end{remark}

\begin{proof}
By $\Gamma\subset\Gamma_{1}$, we obtain $\Delta u>0$. It then suffices to establish the upper bound of $D^{2}u$. For $(x,\xi)\in\Omega\times\mathbb{S}^{n}$, we consider the quantity
\[
Q(x,\xi) = \log u_{\xi\xi}+\vp\big((-u)^{\gamma}|Du|^{2}\big)+\beta\log(-u)+\frac{A}{2}|x|^{2},
\]
where
\[
\vp(t) = -\frac{1}{3}\log\left(3\sup_{\Omega}\left((-u)^{\gamma}|Du|^{2}\right)-t\right), \ \ \gamma = 4\alpha+2,
\]
$\beta$ and $A$ are large constants to be determined later. It is clear that
\begin{equation}\label{vp}
\vp'' = 3(\vp')^{2}.
\end{equation}
Suppose that $(x_{0},\xi_{0})$ is the maximum point of $Q$. We choose the coordinate system such that
\[
u_{ij}(x_{0}) = \lambda_{i}\delta_{ij}, \ \ \lambda_{1} \geq \lambda_{2} \geq \cdots \geq \lambda_{n}.
\]
Then $\xi_{0}=(1,0,\cdots,0)$ and the new quantity
\[
\hat{Q} = \log u_{11}+\vp\big((-u)^{\gamma}|Du|^{2}\big)+\beta\log(-u)+\frac{A}{2}|x|^{2}
\]
still achieves its maximum at $x_{0}$.

\begin{lemma}
At $x_{0}$, we have
\begin{equation}\label{first derivative}
\frac{u_{11i}}{\lambda_{1}}+\vp'\big((-u)^{\gamma}|Du|^{2}\big)_{i}+\frac{\beta}{u}u_{i}+Ax_{i} = 0
\end{equation}
and
\begin{equation}\label{second derivative}
\begin{split}
0 \geq {} & -2\sum_{i>1}\frac{F^{1i,i1}u_{11i}^{2}}{\lambda_{1}}-\frac{F^{ii}u_{11i}^{2}}{\lambda_{1}^{2}}+\frac{(-u)^{\gamma}}{C}F^{ii}\lambda_{i}^{2} \\
& +\vp''F^{ii}\big((-u)^{\gamma}|Du|^{2}\big)_{i}^{2}-\frac{\beta}{u^{2}}F^{ii}u_{i}^{2}+\sum_{i}F^{ii}+\frac{\beta}{u}-C.
\end{split}
\end{equation}
\end{lemma}

\begin{proof}
The maximum principle shows $\hat{Q}_{i}=0$ and $F^{ii}\hat{Q}_{ii}\leq0$ at $x_{0}$. Then we obtain \eqref{first derivative} and
\begin{equation}\label{second derivative eqn 1}
\begin{split}
0 \geq {} & \frac{F^{ii}u_{ii11}}{\lambda_{1}}-\frac{F^{ii}u_{11i}^{2}}{\lambda_{1}^{2}}+\vp'F^{ii}\big((-u)^{\gamma}|Du|^{2}\big)_{ii} \\
& +\vp''F^{ii}\big((-u)^{\alpha}|Du|^{2}\big)_{i}^{2}+\frac{\beta}{u}F^{ii}u_{ii}-\frac{\beta}{u^{2}}F^{ii}u_{i}^{2}+A\sum_{i}F^{ii}.
\end{split}
\end{equation}
For the first term of \eqref{second derivative eqn 1}, differentiating the equation $F(D^{2}u)=1$ with respect to $x_{1}$ twice, using \eqref{derivative of F} and the concavity of $f$,
\begin{equation}\label{second derivative eqn 2}
F^{ii}u_{ii11} = -F^{ij,kl}u_{ij1}u_{kl1}
\geq -\sum_{i\neq j}F^{ij,ji}u_{ij1}^{2}
\geq -2\sum_{i>1}F^{1i,i1}u_{11i}^{2}.
\end{equation}
For the third term of \eqref{second derivative eqn 1}, differentiating the equation $F(D^{2}u)=1$ with respect to $x_{k}$, we obtain $F^{ii}u_{iik}=0$. Combining this with the Cauchy-Schwarz inequality and $F^{ii}\lambda_{i}=1$ (see Proposition \ref{homog}),
\[
\begin{split}
& F^{ii}\left((-u)^{\gamma}|Du|^{2}\right)_{ii} \\
= {} & (-u)^{\gamma}F^{ii}\big(|Du|^{2}\big)_{ii}+2F^{ii}\big((-u)^{\gamma}\big)_{i}\big(|Du|^{2}\big)_{i}+|Du|^{2}F^{ii}\big((-u)^{\gamma}\big)_{ii} \\
= {} & 2(-u)^{\gamma}F^{ii}\lambda_{i}^{2}-4\gamma(-u)^{\gamma-1}F^{ii}\lambda_{i}u_{i}^{2}
+\gamma(\gamma-1)(-u)^{\gamma-2}|Du|^2F^{ii}u_{i}^{2}\\
&-\gamma(-u)^{\gamma-1}|Du|^{2}F^{ii}\lambda_{i} \\
\geq {} & (-u)^{\gamma}F^{ii}\lambda_{i}^{2}-4\gamma^{2}(-u)^{\gamma-2}F^{ii}u_{i}^{4}
+\gamma(\gamma-1)(-u)^{\gamma-2}|Du|^2F^{ii}u_{i}^{2}\\
&-\gamma(-u)^{\gamma-1}|Du|^{2},
\end{split}
\]
Dropping the third term and using $\gamma=4\alpha+2$,
\begin{equation}\label{second derivative eqn 3}
F^{ii}\big((-u)^{\gamma}|Du|^{2}\big)_{ii}
\geq (-u)^{\gamma}F^{ii}\lambda_{i}^{2}-C\sum_{i}F^{ii}-C.
\end{equation}
Substituting \eqref{second derivative eqn 2} and \eqref{second derivative eqn 3} into \eqref{second derivative eqn 1},
\[
\begin{split}
0 \geq {} & -2\sum_{i>1}\frac{F^{1i,i1}u_{11i}^{2}}{\lambda_{1}}-\frac{F^{ii}u_{11i}^{2}}{\lambda_{1}^{2}}+\vp'(-u)^{\gamma}F^{ii}\lambda_{i}^{2} \\
& +\vp''F^{ii}\big((-u)^{\alpha}|Du|^{2}\big)_{i}^{2}+\frac{\beta}{u}F^{ii}\lambda_{i}-\frac{\beta}{u^{2}}F^{ii}u_{i}^{2}+(A-C\vp')\sum_{i}F^{ii}-C\vp'.
\end{split}
\]
Using $F^{ii}\lambda_{i}=1$ (see Proposition \ref{homog}), $C^{-1}\leq\vp'\leq C$ and choosing $A=C+1$, we obtain \eqref{second derivative}.
\end{proof}

We divide the proof into two cases. The constant $\ve$ is a constant to be determined later.

\bigskip
\noindent
{\bf Case 1.} $\lambda_{n}<-\ve\lambda_{1}$.
\bigskip

By \eqref{first derivative} and the Cauchy-Schwarz inequality, we see that
\[
\frac{F^{ii}u_{11i}^{2}}{\lambda_{1}^{2}}
\leq 3(\vp')^{2}F^{ii}\big((-u)^{\gamma}|Du|^{2}\big)_{i}^{2}+\frac{3\beta^{2}}{u^{2}}F^{ii}u_{i}^{2}+3A^{2}F^{ii}x_{i}^{2}.
\]
Plugging this into \eqref{second derivative}, dropping the first term and using $\vp''=3(\vp')^{2}$ (see \eqref{vp}),
\begin{equation}\label{case 1 eqn 2}
\begin{split}
0 \geq {} & \frac{(-u)^{\gamma}}{C}F^{ii}\lambda_{i}^{2}+\left(\vp''-3(\vp')^{2}\right)F^{ii}\big((-u)^{\gamma}|Du|^{2}\big)_{i}^{2} \\[1mm]
& -\frac{\beta+3\beta^{2}}{u^{2}}F^{ii}u_{i}^{2}-3A^{2}F^{ii}x_{i}^{2}+\sum_{i}F^{ii}+\frac{\beta}{u}-C \\
\geq {} & \frac{(-u)^{\gamma}}{C}F^{nn}\lambda_{n}^{2}-\frac{C}{u^{2}}F^{ii}u_{i}^{2}-C\sum_{i}F^{ii}+\frac{C}{u}-C.
\end{split}
\end{equation}
Since $\lambda_{1}\geq\cdots\geq\lambda_{n}$, then (see e.g. \cite{EH,S})
\[
F^{11} \leq F^{22} \leq \cdots \leq F^{nn}.
\]
Combining this with $\lambda_{n}<-\ve\lambda_{1}$,
\[
\frac{(-u)^{\gamma}}{C}F^{nn}\lambda_{n}^{2} \geq \frac{(-u)^{\gamma}}{C}\ve^{2}\lambda_{1}^{2}\sum_{i}F^{ii}.
\]
Substituting this into \eqref{case 1 eqn 2},
\[
0 \geq \frac{(-u)^{\gamma}}{C}\ve^{2}\lambda_{1}^{2}\sum_{i}F^{ii}-\frac{C}{(-u)^{2\alpha+2}}\sum_{i}F^{ii}+\frac{C}{u}.
\]
Without loss of generality, we assume
\[
\frac{(-u)^{\gamma}}{C}\ve^{2}\lambda_{1}^{2} \geq \frac{2C}{(-u)^{2\alpha+2}}+2
\]
and so
\begin{equation}\label{case 1 eqn 1}
\frac{(-u)^{\gamma}}{2C}\ve^{2}\lambda_{1}^{2}\sum_{i}F^{ii}+\sum_{i}F^{ii} \leq \frac{C}{(-u)}.
\end{equation}
Then condition N shows
\[
F^{ii} \geq \frac{1}{N_{1}}\left[\frac{(-u)}{C}\right]^{N_{2}}, \ i=1,\cdots,n.
\]
Combining this with \eqref{case 1 eqn 1}, we obtain $(-u)^{N_{2}+\gamma+1}\lambda_{1}\leq C$. Increasing $\beta$ such that $\beta\geq N_{2}+\gamma+1$, we are done.

\bigskip
\noindent
{\bf Case 2.} $\lambda_{n}\geq-\ve\lambda_{1}$.
\bigskip

Define the index set by
\[
I = \left\{i\in\{1,\cdots,n\} \ | \ F^{ii}\geq \ve^{-1}F^{11} \right\}.
\]
Then \eqref{first derivative} and the Cauchy-Schwarz inequality show
\begin{equation}\label{in I}
\begin{split}
& \sum_{i\in I}\frac{F^{ii}u_{11i}^{2}}{\lambda_{1}^{2}}+\beta\sum_{i\in I}\frac{F^{ii}u_{i}^{2}}{u^{2}} \\
= {} & \sum_{i\in I}\frac{F^{ii}u_{11i}^{2}}{\lambda_{1}^{2}}
+\frac{1}{\beta}\sum_{i\in I}F^{ii}\left(\frac{u_{11i}}{\lambda_{1}}+\vp'\big((-u)^{\gamma}|Du|^{2}\big)_{i}+Ax_{i}\right)^{2} \\
\leq {} & \left(1+\frac{3}{\beta}\right)\sum_{i\in I}\frac{F^{ii}u_{11i}^{2}}{\lambda_{1}^{2}}
+\frac{3}{\beta}(\vp')^{2}\sum_{i\in I}F^{ii}\big((-u)^{\gamma}|Du|^{2}\big)_{i}^{2}+\frac{CA^{2}}{\beta}\sum_{i\in I}F^{ii}
\end{split}
\end{equation}
and
\begin{equation}\label{not in I}
\begin{split}
& \sum_{i\notin I}\frac{F^{ii}u_{11i}^{2}}{\lambda_{1}^{2}}+\beta\sum_{i\notin I}\frac{F^{ii}u_{i}^{2}}{u^{2}} \\
= {} & \sum_{i\notin I}F^{ii}\left(\vp'\big((-u)^{\alpha}|Du|^{2}\big)_{i}+\frac{\beta}{u}u_{i}+Ax_{i}\right)^{2}+\beta\sum_{i\notin I}\frac{F^{ii}u_{i}^{2}}{u^{2}} \\
\leq {} & 3(\vp')^{2}\sum_{i\in I}F^{ii}\big((-u)^{\gamma}|Du|^{2}\big)_{i}^{2}
+(3\beta^{2}+\beta)\sum_{i\notin I}\frac{F^{ii}u_{i}^{2}}{u^{2}}+3A^{2}\sum_{i\notin I}F^{ii}x_{i}^{2} \\
\leq {} & 3(\vp')^{2}\sum_{i\in I}F^{ii}\big((-u)^{\gamma}|Du|^{2}\big)_{i}^{2}+\left[\frac{C(3\beta^{2}+\beta)}{(-u)^{2\alpha+2}}+CA^{2}\right]\frac{F^{11}}{\ve},
\end{split}
\end{equation}
where we used $F^{ii}<\ve^{-1}F^{11}$ for $i\notin I$ in the last line. Substituting \eqref{in I} and \eqref{not in I} into \eqref{second derivative}, we obtain
\begin{equation}\label{case 2 eqn 1}
\begin{split}
0 \geq {} & -2\sum_{i>1}\frac{F^{1i,i1}u_{11i}^{2}}{\lambda_{1}}
-\left(1+\frac{3}{\beta}\right)\sum_{i\in I}\frac{F^{ii}u_{11i}^{2}}{\lambda_{1}^{2}}
+\frac{(-u)^{\gamma}}{C}F^{ii}\lambda_{i}^{2} \\
& -\bigg[\frac{C(3\beta^{2}+\beta)}{(-u)^{2\alpha+2}}+CA^{2}\bigg]\frac{F^{11}}{\ve}
+\left(\vp''-3(\vp')^{2}\right)F^{ii}\big((-u)^{\gamma}|Du|^{2}\big)_{i}^{2}   \\
& +\left(1-\frac{CA^{2}}{\beta}\right)\sum_{i}F^{ii}+\frac{C}{u}.
\end{split}
\end{equation}
Without loss of generality, we assume that
\begin{equation*}
\frac{(-u)^{\gamma}}{C}\lambda_{1}^{2} \geq \frac{1}{\ve}\bigg[\frac{C(3\beta^{2}+\beta)}{(-u)^{2\alpha+2}}+CA^{2}\bigg].
\end{equation*}
Choose $\beta$ sufficiently large such that $\frac{CA^{2}}{\beta}\leq\frac{1}{2}$ and recall $\vp''=3(\vp')^{2}$ (see \eqref{vp}). Then \eqref{case 2 eqn 1} becomes
\begin{equation}\label{case 2 eqn 2}
0 \geq -2\sum_{i>1}\frac{F^{1i,i1}u_{11i}^{2}}{\lambda_{1}}
-\frac{3}{2}\sum_{i\in I}\frac{F^{ii}u_{11i}^{2}}{\lambda_{1}^{2}}+\frac{(-u)^{\gamma}}{C}\sum_{i>1}F^{ii}\lambda_{i}^{2}+\frac{1}{2}\sum_{i}F^{ii}+\frac{C}{u}.
\end{equation}
Using $1\notin I$ and $-F^{1i,i1}=\frac{F^{ii}-F^{11}}{\lambda_{1}-\lambda_{i}}$ (see \eqref{derivative of F}),
\[
\begin{split}
& -2\sum_{i>1}\frac{F^{1i,i1}u_{11i}^{2}}{\lambda_{1}}
-\frac{3}{2}\sum_{i\in I}\frac{F^{ii}u_{11i}^{2}}{\lambda_{1}^{2}} \\
\geq {} & -2\sum_{i\in I}\frac{F^{1i,i1}u_{11i}^{2}}{\lambda_{1}}
-\frac{3}{2}\sum_{i\in I}\frac{F^{ii}u_{11i}^{2}}{\lambda_{1}^{2}} \\
= {} & 2\sum_{i\in I}\frac{(F^{ii}-F^{11})u_{11i}^{2}}{\lambda_{1}(\lambda_{1}-\lambda_{i})}
-\frac{3}{2}\sum_{i\in I}\frac{F^{ii}u_{11i}^{2}}{\lambda_{1}^{2}}.
\end{split}
\]
For $i\in I$, we have $F^{11}\leq\ve F^{ii}$. Combining this with $\lambda_{i}\geq\lambda_{n}\geq-\ve\lambda_{1}$, we obtain
\[
2\sum_{i\in I}\frac{(F^{ii}-F^{11})u_{11i}^{2}}{\lambda_{1}(\lambda_{1}-\lambda_{i})}
\geq \frac{2(1-\ve)}{1+\ve}\sum_{i\in I}\frac{F^{ii}u_{11i}^{2}}{\lambda_{1}^{2}}.
\]
Choosing $\ve$ sufficiently small, we obtain
\begin{equation}\label{third order terms}
\begin{split}
& -2\sum_{i>1}\frac{F^{1i,i1}u_{11i}^{2}}{\lambda_{1}}
-\frac{3}{2}\sum_{i\in I}\frac{F^{ii}u_{11i}^{2}}{\lambda_{1}^{2}}
\geq \left(\frac{2(1-\ve)}{1+\ve}-\frac{3}{2}\right)\sum_{i\in I}\frac{F^{ii}u_{11i}^{2}}{\lambda_{1}^{2}} \geq 0.
\end{split}
\end{equation}
By \eqref{case 2 eqn 2} and \eqref{third order terms}, we see that
\[
\frac{(-u)^{\gamma}}{C}F^{ii}\lambda_{i}^{2}+\frac{1}{2}\sum_{i}F^{ii} \leq \frac{C}{(-u)}.
\]
Combining this with condition N and arguing as in Case 1, we are done.
\end{proof}

\section{Liouville theorem}\label{Liouville theorem}
By Remark \ref{condition N remark}, Proposition \ref{p-1} and \ref{khess}, Corollaries \ref{Liouville k Hessian} and \ref{Liouville p MA} follow immediately from Theorem \ref{lio}. To prove Theorem \ref{lio}, we follow the argument of \cite{TW}.

\begin{proof}
Adding a constant if necessary we may assume that $u(0)=0$. We fix a sufficiently large radius $R\geq 1$. Consider the function and the set
\[
v_R(y):=\frac{u(Ry)-R^2}{R^2}, \ \ \Om_R := \lbrace y\in\mathbb R^n|\ v_R(y)<0\rbrace.
\]
It then follows from $D_{y}^{2}v_{R}=D_{x}^{2}u$ that
\begin{equation}\label{Liouville v R}
\begin{cases}
\ F(D^{2}v_{R}) = 1 & \mbox {in $\Omega_{R}$}, \\
\ v_{R} = 0 & \mbox{on $\de\Omega_{R}$}.
\end{cases}
\end{equation}
Using the quadratic growth condition of $u$ (i.e. $u(x)\geq C^{-1}|x|-C$), we have
$$\Om_R\subset\lbrace y\in\mathbb R^n \ |\ C^{-1}R^2|y|^2-C<R^2\rbrace\subset\lbrace y\in\mathbb R^n\ |\ |y|^2\leq C(C+1)\rbrace,$$
which implies
\[
\mathrm{diam}(\Omega_{R}) \leq 2\sqrt{C(C+1)}.
\]
Applying Theorem \ref{intgrad} to \eqref{Liouville v R} on each component of $\Omega_{R}$,
\begin{equation}\label{Liouville eqn 1}
(-v_{R})^{\beta}|D^{2}v_{R}| \leq C \ \text{in $\Omega_{R}$}.
\end{equation}
Consider the set
$$\Om'_R:=\lbrace y\in\mathbb R^n \,|\, v_R\leq -\frac12\rbrace=\lbrace y\in\mathbb R^n \,|\, u(Ry)\leq\frac12R^2\rbrace.$$
Then \eqref{Liouville eqn 1} shows
\[
|D^{2}v_{R}| \leq 2^{\beta}C \ \text{in $\Omega_{R}'$}
\]
and so
\[
|D^{2}u| \leq 2^{\beta}C \ \text{in $\tilde{\Omega}_{R}:=\lbrace x\in\mathbb R^n |\ u(x)\leq \frac12R^2\rbrace$}.
\]
Since $R$ is arbitrary, then
\[
\|D^{2}u\|_{C^{0}(\mathbb{R}^{n})} \leq \Lambda
\]
The set $S_{\Lambda}:=\{\lambda\in\Gamma\,|\,f(\lambda)=1,\,|\lambda|\leq\Lambda\}$ is a compact subset of $\Gamma$. There exists a constant $K$ depending only on $\Gamma$, $f$ and $\Lambda$ such that
\[
K^{-1} \leq f_{i}\leq K \  \text{in $S_{\Lambda}$}.
\]
By \eqref{derivative of F}, the above shows that $F(D^{2}u)=1$ is uniformly elliptic in $\mathbb{R}^{n}$. Finally Evans-Krylov theory (see e.g. \cite{GT}) implies that for some $\gamma\in(0,1)$, we have
\[
\|D^2u\|_{C^{\gamma}(\mathbb{R}^{n})}
= \lim_{R\to\infty}\|D^2u\|_{C^{\gamma}(B_{R})}
\leq \lim_{R\to\infty}\frac{C\|D^2u\|_{C^{\gamma}(B_{2R})}}{R^{\gamma}} \leq \lim_{R\to\infty}\frac{C\Lambda}{R^{\gamma}} = 0.
\]
This shows that $D^2u$ is constant and $u$ has to be a quadratic polynomial.
\end{proof}

\end{document}